\documentclass[9pt,journal]{IEEEtran}

\usepackage{amsmath,amsfonts,amssymb,mathtools}
\usepackage{hhline}
\usepackage{cite}
\usepackage{graphicx}
\usepackage{psfrag}
\usepackage{subfigure}
\usepackage{stfloats}
\usepackage{hyperref}
\usepackage{clrscode}
\usepackage{array}
\DeclareMathAlphabet\mathbfcal{OMS}{cmsy}{b}{n}

\newtheorem{Lm}{Lemma}

\newtheorem{remark}{Remark}
\newtheorem{exmp}{Example}

\newcommand{\E}[1]{\mathop{{\rm \bf E}\left\{#1\right\}}\nolimits}
\newcommand{\Diff}[1]{\mathop{{\rm \Delta}\left[#1\right]}\nolimits}

\def\url#1{\expandafter\string\csname #1\endcsname}

\begin{document}
\title{Chandrasekhar-based Maximum Correntropy Kalman Filtering with the Adaptive Kernel Size Selection}

\author{Maria V.~Kulikova
\thanks{Manuscript received ??; revised ??.
The author acknowledges the financial support of the Portuguese FCT~--- \emph{Funda\c{c}\~ao para a Ci\^encia e a Tecnologia}, through the project UID/Multi/04621/2019 of CEMAT/IST-ID, Center for Computational and Stochastic Mathematics, Instituto Superior T\'ecnico, University of Lisbon.}
\thanks{The author is with
CEMAT (Center for Computational and Stochastic Mathematics), Instituto Superior T\'{e}cnico, Universidade de Lisboa,
          Av. Rovisco Pais 1,  1049-001 LISBOA, Portugal;
          Email: maria.kulikova@ist.utl.pt}}

\markboth{PREPRINT}{}

\maketitle

\begin{abstract}
This technical note is aimed to derive the Chandrasekhar-type recursion for the maximum correntropy criterion (MCC) Kalman filtering (KF). For the classical KF, the  first Chandrasekhar difference equation was proposed at the beginning of 1970s. This is the alternative to the traditionally used Riccati recursion and it yields the so-called fast implementations known as the Morf-Sidhu-Kailath-Sayed KF algorithms. They are proved to be computationally cheap because of propagating the matrices of a smaller size than $n \times n$ error covariance matrix in the Riccati recursion. The problem of deriving the Chandrasekhar-type recursion within the MCC estimation methodology has never been raised yet in engineering literature. In this technical note, we do the first step and derive the Chandrasekhar MCC-KF estimators for the case of adaptive kernel size selection strategy, which implies a constant scalar adjusting weight. Numerical examples substantiate a practical feasibility of the newly suggested MCC-KF implementations and correctness of the presented theoretical derivations.
\end{abstract}


\IEEEpeerreviewmaketitle

\section{Introduction}

The maximum correntropy criterion (MCC) filtering/smoothing has become an important topic for an analysis in the past few years, both for linear~\cite{2007:Liu,2014:Chen,2015:Chen,2011:Cinar,2012:Cinar,2016:Izanloo,2017:Chen,2017:Liu:KF,2017:CSL:Kulikova} and nonlinear systems~\cite{2016:Liu:EKF,2018:Kulikov:SP,2017:Liu:UKF,2017:Qin,2017:Wang,2016:Wang}. In engineering literature, the MCC Kalman-like estimators were proved to be robust with respect to outliers/impulsive noises and to outperform the classical Kalman filtering (KF) for estimation accuracy in case of non-Gaussian uncertainties in state-space models. The problem of designing the ``distributionally robust'' filtering/smoothing methods has a long history~\cite{1964:Huber}. The related problem of constructing the statistically valid uncertainty bounds has been studied in~\cite{1984:Spall,1995:Spall,2004:Maryak}. Apart from the MCC-KF methodology examined in this paper, we mention a few other strategies for detecting the outliers. These are the Huber-based and M-estimator-based KF algorithms suggested in~\cite{1977:Masreliez,2010:Hajiyev,2013:Chang}, the unknown input filtering (UIF) methodology proposed to model the unknown external excitations as unknown inputs and to derive the robust observer in~\cite{2014:Charandabi}, and many other estimation strategies. We may also note the most recent and comprehensive survey of the existed Kalman-like {\it smoothing} methods developed for the non-Gaussian state-space models in~\cite{2017:Aravkin}. In this paper, we focus on the Kalman filtering under the MCC methodology.

Previous research on the MCC-KF implementation methods has produced the Riccati recursion-based algorithms, only. However, it is worth noting here that the classical KF allows for an alternative filter mechanization suggested at the beginning of 1970s in~\cite{1973:Kailath}. It implies the so-called Chandrasekhar recursion and yields the class of the fast KF implementations known as the Morf-Sidhu-Kailath-Sayed algorithms~\cite{1974:Morf,1975:Morf,1994:Sayed}. To the best of author's knowledge, the question about possibility to derive the Chandrasekhar-type recursion under the MCC filtering strategy has never been raised before. This is the challenge to be addressed in this technical note.

The essential starting point for further discussion is to review the Chandrasekhar-type algorithms existed for the classical KF in engineering literature. The first fast KF implementations were derived for linear time-invariant systems in~\cite{1973:Kailath}. The key idea behind the Chandrasekhar-based filtering is to propagate the difference between the error covariance matrices $\Delta_k = P_{k} - P_{k-1}$ instead of updating the matrix $P_{k} \in {\mathbb R}^{n\times n}$ through the classical Riccati recursion. It is motivated by the fact that the rank of matrix $\Delta_k$ is often less than the rank of $P_{k}$, especially for time-invariant state-space models, i.e. for the systems with constant (over time) matrices characterized the model. As a result, a low-rank approximation of the difference matrix $\Delta_k$ can be utilized and the obtained filleting algorithms update the matrices of a smaller size than the number of states $n$ to be estimated. Thus, the Chandrasekhar-type algorithms are, in general, computationally cheaper than the Riccati-based KF implementations. To achieve our goal, the Riccati recursion required for propagating the covariance matrix $P_k$ should be mathematically re-formulated in terms of the difference matrix $\Delta_k$ propagation. Having done this step, one receives the Chandrasekhar-type recursion. Recall, at the first time, this problem was solved for the case of constant {\it continuous-time} systems in~\cite{1973:Kailath}. The derived differential equations turned out to be similar to a certain equation introduced by the astrophysicist S.~Chandrasekhar in 1948 for solving finite-interval Wiener-Hopf equation, and the obtained KF recursion was called of Chandrasekhar type. Almost at the same time, the problem was solved for constant {\it discrete-time} systems as well in~\cite{1974:Morf}. However, the related solution for {\it time-variant} systems is rather complicated and it was derived twenty years later in~\cite{1994:Sayed}. Nowadays, there exists a wide variety of the Chandrasekhar-based KF implementations. For instance, we may mention the robust factored-form (square-root) filtering/smoothing methods in~\cite{1975:Morf,1995:Park:smoothing} as well as the $H_{\infty}$ filtering algorithms in~\cite{1994:Hassibi,2000:Hassibi}. Meanwhile, the Chandrasekhar-type recursion under the MCC filtering methodology has never been derived, so far. In other words, all existed MCC-KF implementation methods involve the Riccati-type recursion for propagating the error covariance matrix.

For the stated problem, a number of questions arise: (i) is it possible to derive Chandrasekhar-type recursion for the Kalman-like filtering under the examined MCC methodology? (ii) If the recursion exists, does it allow for designing the related low-rank implementations for the MCC-KF estimator? In this paper, we answer positively for both questions and suggest a solution for the case of the MCC-KF adaptive kernel size selection strategy with a constant scalar adjusting weight. The results of numerical experiments substantiate the theoretical derivations presented in this paper and prove a practical feasibility of the suggested MCC-KF Chandrasekhar-type  implementations.

\section{Riccati recursion-based Maximum Correntropy Kalman Filtering}

Consider a linear discrete-time stochastic system
  \begin{align}
   x_{k+1} = & F x_{k}  + G w_{k}, \label{eq:st:1} \\
    y_k  = & H x_k + v_k , \quad k \ge 0  \label{eq:st:2}
  \end{align}
where the system matrices $F \in \mathbb R^{n\times n}$, $G \in \mathbb R^{n\times q}$ and $H \in \mathbb R^{m\times n}$ are known and constant over time. The vectors $x_k \in \mathbb R^n$ and $y_k \in \mathbb R^m$ are the unknown dynamic state and available measurements, respectively. The random variables $x_0$, $w_k$ and $v_k$ satisfy
\begin{align*}
&\E{x_0}       = \bar x_0,    &  &\E{(x_0-\bar x_0)(x_0-\bar x_0)^{\top}}  = \Pi_0, \\
&\E{w_k}       = \E{v_k} = 0, &  &\E{w_kx_0^{\top}}  = \E{v_kx_0^{\top}}  = 0,  \\
&\E{w_kv_k^{\top}}  = 0,           &  &\E{w_kw_j^{\top}} = Q\delta_{kj},   \E{v_kv_j^{\top}}  = R\delta_{kj}
\end{align*}
where the covariance matrices $Q \in \mathbb R^{q\times q}$ and $R \in \mathbb R^{m\times m}$ are known. The symbol $\delta_{kj}$ is the Kronecker delta function.

The classical KF produces the minimum {\it linear} expected mean square error (MSE) estimate $\hat x_{k|k}$ of the state vector $x_k$, given the measurements $\mathcal{Y}_0^k = \{y_0, \ldots, y_k \}$.
The estimator can be formulated in the {\it a priori} form as proposed in~\cite[Theorem~9.2.1]{2000:Kailath:book}, i.e. when the one-step ahead predicted estimate $\hat x_{k+1|k}$ ({\it a priori} estimate) is propagated as follows:
\begin{align}
\hat x_{k+1|k}  & =   F \hat x_{k|k-1}+K_{p,k}e_k, & e_k & = y_k-H\hat x_{k|k-1},  \label{KF:X:Riccati} \\
K_{p,k}  & = FP_{k|k-1}H^{\top}R_{e,k}^{-1}, & R_{e,k} & = R+HP_{k|k-1}H^{\top} \label{KF:K:Riccati}
\end{align}
where the innovations are defined as $e_k = y_k-H\hat x_{k|k-1} $ with the covariance matrix $R_{e,k} = \E{e_ke_k^{\top}}$, and $K_{k}=\E{\hat x_{k+1|k} e_k^{\top}} = FP_{k|k-1}H^T$. The matrix $P_{k|k-1}$ is the one-step ahead predicted error covariance $P_{k|k-1}=\E{ (x_{k}-\hat
x_{k|k-1})(x_{k}-\hat x_{k|k-1})^{\top}}$ propagated through the Riccati difference recursion:
\begin{equation}
P_{k+1|k}  = FP_{k|k-1}F^{\top} + GQG^{\top} - K_{p,k}R_{e,k}K_{p,k}^{\top} \label{KF:P:Riccati}
\end{equation}
with the initial values $P_{0|-1} =  \Pi_0 > 0$ and $\hat x_{0|-1}  = \bar x_0$.

Being a {\it linear} estimator, the classical KF exhibits only sub-optimal behavior in non-Gaussian settings. To enhance its estimation quality and robustness with respect to outliers (impulsive noise), the KF {\it linear} expected MSE estimation criterion has been combined with the maximum correntropy approach in~\cite{2012:Cinar,2016:Izanloo,2017:Chen}. More precisely, the concept of {\it correntropy} represents a similarity measure of two random variables~\cite{2007:Liu}. It can be used as an optimization cost in the related estimation problem as discussed in~\cite[Chapter~5]{2018:Principe:book}: an estimator of unknown state $X \in {\mathbb R}$ can be defined as a function of observations $Y \in {\mathbb R}^m$, i.e. $\hat X = g(Y)$ where $g$ is solved by maximizing the correntropy between $X$ and $\hat X$, which is defined as follows~\cite{2012:Chen}:
\begin{equation} \label{mcc:kriterion}
g_{MCC} = \mbox{arg}\max \limits_{g \in G} V(X,\hat X) = \mbox{arg}\max \limits_{g \in G} \E{k_{\sigma}\Bigl(X - g(Y)\Bigr)}
\end{equation}
where $G$ stands for the collection of all measurable functions of $Y$, $k_{\sigma}(\cdot)$ is a kernel function and  $\sigma > 0$ is the kernel size (bandwidth). One of the most popular kernel function utilized in practice is the Gaussian kernel given as follows:
\begin{equation}\label{Gauss_kernel}
k_{\sigma}(X - \hat X) = \exp \left\{ -{(X - \hat X)^2}/{(2\sigma^2)}\right\}.
\end{equation}
It is not difficult to see that the MCC cost~\eqref{mcc:kriterion} with Gaussian kernel~\eqref{Gauss_kernel} reaches its maximum if and only if $X = \hat X$.

In~\cite{2016:Izanloo}, the MCC-KF is developed by solving the following estimation problem with the Gaussian kernel:
\begin{align}
\hat x_{k|k} & = \mbox{arg}\max J(k), \\
 J(k)  &= k_{\sigma}(\|\hat x_{k|k}-F\hat x_{k-1|k-1}\|)  + k_{\sigma}(\|y_k-H\hat x_{k|k}\|).
\end{align}

Next, a fixed point rule (with one iterate, only) has been used for solving (with respect to $\hat x_{k|k}$) the resulted nonlinear equation arisen in the optimization problem above. It results to the following filtering recursion for the state~\cite[p.~503]{2016:Izanloo}:
\begin{equation}
\hat x_{k|k}  =  F\hat x_{k-1|k-1}  + K_k^{\lambda}(y_k - H\hat x_{k|k-1}) \label{eq:result}
\end{equation}
where the gain matrix is proved to be computed as follows~\cite{2016:Izanloo}: 
$K_k^{\lambda} = \lambda_k \left( P_{k|k-1}^{-1}+\lambda_k H^{\top} R^{-1} H \right)^{-1}H^{\top} R^{-1}$ and the scalar adjusting weight $\lambda_k$ is given by
\begin{equation}
\lambda_{k}  = \frac{k_{\sigma}(\|y_k-H\hat x_{k|k-1}\|_{R^{-1}})}{k_{\sigma}(\|\hat x_{k|k-1}-F\hat x_{k-1|k-1}\|_{P_{k|k-1}^{-1}})}.  \label{eq:lambda}
\end{equation}

Finally, the recursion for the state estimate in~\eqref{eq:result} is utilized
 with a symmetric Joseph stabilized equation existed for the classical KF for the error covariance matrix $P_{k|k}$ calculation; see~\cite{2006:Simon:book,2015:Grewal:book}. The resulted computational method was called the MCC-KF estimator and it is summarized in~\cite[p.~503]{2016:Izanloo}. Further, the estimation quality of the original MCC-KF method has been boosted by deriving the mathematically equivalent formulas for the gain matrix and error covariance, which are similar to the classical KF equations presented in~\cite[pp.~128-129]{2006:Simon:book}.
 This approach yields the so-called improved MCC-KF (IMCC-KF) estimator suggested in~\cite{2017:CSL:Kulikova}. It consists of the following steps.

\textsc{Time Update} ($k=1, \ldots, K$). At this stage, the one-step ahead predicted ({\it a priori}) estimate, $\hat x_{k|k-1}$, is computed together with the corresponding error covariance matrix $P_{k|k-1}$ as follows:
\begin{align}
\hat x_{k|k-1} & = F\hat x_{k-1|k-1}, \label{mcckf:p:X} \\
P_{k|k-1} & = FP_{k-1|k-1}F^{\top}+GQG^{\top}. \label{mcckf:p:P}
\end{align}

\textsc{Measurement Update}  ($k=1, \ldots, K$).  The correction step is called the measurement update where the {\it a posteriori} estimate $\hat x_{k|k}$ is calculated together with the corresponding error covariance matrix $P_{k|k}$ as follows:
\begin{align}
K_{k}^{\lambda} & = \lambda_k P_{k|k-1}H^{\top}[R_{e,k}^{\lambda}]^{-1}\!\!,  \! \! \! & R_{e,k}^{\lambda} & = \lambda_kHP_{k|k-1}H^{\top}+R, \label{mcckf:f:K} \\
\hat x_{k|k} & =    \hat x_{k|k-1}+K_{k}^{\lambda}e_k,\! \! \!           & e_k & = y_k-H\hat x_{k|k-1},   \label{mcckf:f:X} \\
P_{k|k}  & = (I - K^{\lambda}_{k}H)P_{k|k-1} && \label{mcckf:f:P}
\end{align}
where the scaling (inflation) parameter $\lambda_{k}$ is computed by~\eqref{eq:lambda}.

The IMCC-KF estimator can be re-formulated in the {\it a priori} form similar to the classical KF equations~\eqref{KF:X:Riccati}~-- \eqref{KF:P:Riccati}. Having substituted~\eqref{mcckf:f:X} into~\eqref{mcckf:p:X}, we get the recursion for the one-step ahead predicted ({\it a priori}) estimate $\hat x_{k+1|k}$ as follows:
\begin{align}
\hat x_{k+1|k}      & =   F \hat x_{k|k-1}+K_{p,k}^{\lambda}e_k,  \label{MCCKF:X:Riccati} \\
K_{p,k}^{\lambda}   & = \lambda_kFP_{k|k-1}H^{\top}[R_{e,k}^{\lambda}]^{-1}   \label{MCCKF:K:Riccati}
\end{align}
where $R_{e,k}^{\lambda}$ is defined by~\eqref{mcckf:f:K} and the {\it a priori} error covariance $P_{k+1|k}$ is computed via the following Riccati-type recursion:
\begin{equation}
P_{k+1|k}  = FP_{k|k-1}F^{\top}+GQG^{\top} - \frac{1}{\lambda_k}K_{p,k}^{\lambda}R_{e,k}^{\lambda}[K_{p,k}^{\lambda}]^{\top}. \label{MCCKF:P:Riccati}
\end{equation}

Indeed, having substituted formula~\eqref{mcckf:f:P} into~\eqref{mcckf:p:P} and taking into account the symmetric form of any covariance matrix and the fact  that $K^{\lambda}_{p,k} = FK^{\lambda}_{k}$, we prove~\eqref{MCCKF:P:Riccati} as follows:
\begin{align*}
P_{k+1|k} & = F(I - K^{\lambda}_{k}H)P_{k|k-1}F^{\top}+GQG^{\top} \\
& = FP_{k|k-1}F^{\top} + GQG^{\top} - FK^{\lambda}_{k}HP_{k|k-1}F^{\top}  \\
& = FP_{k|k-1}F^{\top} + GQG^{\top} - K^{\lambda}_{p,k}(FP_{k|k-1}H^{\top})^{\top} \\
& = FP_{k|k-1}F^{\top} + GQG^{\top} - \frac{1}{\lambda_k}K^{\lambda}_{p,k}R_{e,k}^{\lambda}[K^{\lambda}_{p,k}]^{\top}
\end{align*}

Although $\lambda_k$ is a scalar value, it is preferable to avoid the division. Thus, equations~\eqref{MCCKF:X:Riccati}~-- \eqref{MCCKF:P:Riccati} can be written as follows:
\begin{align}
\hat x_{k+1|k} & =   F \hat x_{k|k-1} + \lambda_kK_{p,k}e_k, \label{MCCKF:X:Riccati1} \\
K_{p,k} & = FP_{k|k-1}H^{\top}[R_{e,k}^{\lambda}]^{-1} \label{MCCKF:K:Riccati1} \\
P_{k+1|k}      & = FP_{k|k-1}F^{\top}+GQG^{\top} - \lambda_kK_{p,k}R_{e,k}^{\lambda}K_{p,k}^{\top}. \label{MCCKF:P:Riccati1}
\end{align}

\section{Chandrasekhar-based Maximum Correntropy Kalman Filtering} \label{sec:derivation}

For a better presentation of a new material, we introduce the backward difference operator $\Diff{\cdot}$ that means $\Diff{A_k} = A_k - A_{k-1}$ for any matrix $A$. The Riccati-based filtering implies the error covariance matrix $P_{k+1|k} \in {\mathbb R}^{n\times n}$ propagation through the classical KF recursion~\eqref{KF:P:Riccati} as well as the IMCC-KF equation~\eqref{MCCKF:P:Riccati1}. The key idea of the Chandrasekhar recursion-based filtering is to propagate the difference $\Diff{P_{k+1|k}} = P_{k+1|k} - P_{k|k-1}$ instead of $P_{k+1|k}$. It is motivated by the fact that the rank of matrix $\Diff{P_{k+1|k}} \in {\mathbb R}^{n\times n}$ is often less than $n$, which is a number of states to be estimated by the filter, especially for time-invariant state-space models (i.e. with constant system matrices). Hence, the difference matrix $\Diff{P_{k+1|k}}$ implies a low-rank approximation and, thus, the related filtering algorithms can be derived for propagating the resulted lower rank factors. More precisely, the fast Morf-Sidhu-Kailath-Sayed algorithms derived for the classical KF are based on a factorization (non-uniquely defined) of symmetric indefinite matrix $\Diff{P_{1|0}}$ in the form $\Diff{P_{1|0}} = L_0M_0L_0^{\top}$ where $L_0 \in {\mathbb R}^{n\times\alpha}$, $M_0 \in {\mathbb R}^{\alpha\times\alpha}$, and $\alpha$ is called  the displacement rank. Similar, for the IMCC-KF recursion in~\eqref{MCCKF:P:Riccati1}, we have
\begin{align*}
\alpha & = {\rm rank}\Diff{P_{1|0}} = {\rm rank} (P_{1|0}-P_{0|-1}) \\
       & = {\rm rank} (F\Pi_0F^{\top} + GQG^{\top} - \lambda_0K_{p,0}R_{e,0}^{\lambda}K_{p,0}^{\top} - \Pi_0).
\end{align*}

It is important to acknowledge that the mentioned factorization is performed only once, i.e. at the initial step of any Chandrasekhar-based filtering method. When the resulted factors $L_0 \in {\mathbb R}^{n\times\alpha}$ and $M_0 \in {\mathbb R}^{\alpha\times\alpha}$ are defined, they are propagated instead of entire matrices $P_{k+1|k}$ and/or $\Diff{P_{k+1|k}}$. In general, the displacement rank $\alpha \le n$ and, hence, the Chandrasekhar-type algorithms for propagating $M_k \in {\mathbb R}^{\alpha\times\alpha}$ and $L_k \in {\mathbb R}^{n\times\alpha}$ are computationally cheaper than the Riccati-based implementations that update a full rank error covariance matrix $P_{k+1|k}$ of size $n$. The computational complexity is shown to be reduced from $O(n^3)$ related to the Riccati recursion to $O(n^2\alpha)$ related to the Chandrasekhar recursion per iteration~\cite{1974:Morf,1994:Sayed}. The methods for implementing the underlying low-rank approximation of $\Diff{P_{1|0}}$ will be discussed at the end of this section. The first aim of this section is to derive the Chandrasekhar recursion for the IMCC-KF filtering.  In other words, the corresponding IMCC-KF Riccati-type recursion in~\eqref{MCCKF:P:Riccati1} should be re-formulated in terms of the matrices $\Diff{P_{k+1|k}}$, i.e. equation~\eqref{MCCKF:P:Riccati1} and the newly derived Chandrasekhar-based formulas should be mathematically equivalent.  We prove the following theoretical result.

\begin{Lm} \label{lemma:1} The IMCC-KF Riccati-based filtering formulas~\eqref{MCCKF:X:Riccati1}~-- \eqref{MCCKF:P:Riccati1} with a constant adjusting parameter $\lambda$ are equivalent to  the following Chandrasekhar-type recursions:
\begin{align}
\!\!\!\!\!\Diff{P_{k+1|k}} & = \Bigl(F-\lambda K_{p,k}H\Bigr)\Bigl(\Diff{P_{k|k-1}}  \nonumber \\
 &  \phantom{=} +\lambda \Diff{P_{k|k-1}}H^{\top} [R_{e,k-1}^{\lambda}]^{-1} H\Diff{P_{k|k-1}}\Bigr)   \nonumber \\
 & \phantom{=} \times \Bigl(F-\lambda K_{p,k}H\Bigr)^{\top}, \label{lemma:1} \\
\Diff{P_{k+1|k}}  & =  \Bigl(F-\lambda K_{p,k-1}H\Bigr)\Bigl(\Diff{P_{k|k-1}}  \nonumber \\
 &  \phantom{=} -\lambda  \Diff{P_{k|k-1}}H^{\top} [R_{e,k}^{\lambda}]^{-1} H\Diff{P_{k|k-1}}\Bigr)   \nonumber \\
 & \phantom{=} \times \Bigl(F-\lambda K_{p,k-1}H\Bigr)^{\top}. \label{lemma:2}
\end{align}
\end{Lm}

\begin{proof} First, from expression~\eqref{mcckf:f:K} for $R_{e,k}^{\lambda}$ we derive
\[
\Diff{R_{e,k}^{\lambda}}  = \lambda H\Diff{P_{k|k-1}}H^{\top}
\]
and, hence
\begin{equation}
R_{e,k-1}^{\lambda}  =  R_{e,k}^{\lambda} - \Diff{R_{e,k}^{\lambda}} = R_{e,k}^{\lambda} - \lambda H\Diff{P_{k|k-1}}H^{\top}. \label{proof:Rekminus1}
\end{equation}
Next, for the gain matrix $K_{p,k}$ in~\eqref{MCCKF:K:Riccati1}, we get
\begin{align}
\Diff{K_{p,k}R_{e,k}^{\lambda}}  & = K_{p,k}R_{e,k}^{\lambda} - K_{p,k-1}R_{e,k-1}^{\lambda} \nonumber \\
& = F\Diff{P_{k|k-1}}H^{\top} \label{proof:new}
\end{align}
and, hence
\begin{equation}
K_{p,k} = \left[K_{p,k-1}R_{e,k-1}^{\lambda} + F\Diff{P_{k|k-1}}H^{\top}\right] \left[R_{e,k}^{\lambda} \right]^{-1}. \label{proof:Kpk}
\end{equation}

From the IMCC-based Riccati-type recursion in~\eqref{MCCKF:P:Riccati1}, we obtain
\begin{align}
\Diff{P_{k+1|k}} & = P_{k+1|k} - P_{k|k-1} = F\Diff{P_{k|k-1}}F^{\top}  \nonumber \\
                 & \phantom{=} - \lambda K_{p,k}R_{e,k}^{\lambda}K_{p,k}^{\top} + \lambda K_{p,k-1}R_{e,k-1}^{\lambda}K_{p,k-1}^{\top} \nonumber \\
                 & = F\Diff{P_{k|k-1}}F^{\top}  - \lambda \Diff{K_{p,k}R_{e,k}^{\lambda}K_{p,k}^{\top}}. \label{proof:1}
\end{align}

To derive the expression for the last term in equation~\eqref{proof:1}, i.e. for $\Diff{K_{p,k}R_{e,k}^{\lambda}K_{p,k}^{\top}}$, one multiplies formula~\eqref{proof:new} by $K_{p,k}^{\top}$ and $K_{p,k-1}^{\top}$ values, respectively. Thus, we have
\begin{align*}
\Diff{K_{p,k}R_{e,k}^{\lambda}}K_{p,k}^{\top} & = F\Diff{P_{k|k-1}}H^{\top}K_{p,k}^{\top} \\
 & = K_{p,k}R_{e,k}^{\lambda}K_{p,k}^{\top}-K_{p,k-1}R_{e,k-1}^{\lambda}K_{p,k}^{\top}, \\
\Diff{K_{p,k}R_{e,k}^{\lambda}}K_{p,k-1}^{\top} & = F\Diff{P_{k|k-1}}H^{\top}K_{p,k-1}^{\top} \\
 & = K_{p,k}R_{e,k}^{\lambda}K_{p,k-1}^{\top}-K_{p,k-1}R_{e,k-1}^{\lambda}K_{p,k-1}^{\top}
\end{align*}
and, next, we summarize
\begin{align*}
& \Diff{K_{p,k}R_{e,k}^{\lambda}}K_{p,k}^{\top} + \Diff{K_{p,k}R_{e,k}^{\lambda}}K_{p,k-1}^{\top} \\
& = \Diff{K_{p,k}R_{e,k}^{\lambda}K_{p,k}^{\top}}  + K_{p,k}R_{e,k}^{\lambda}K_{p,k-1}^{\top} - K_{p,k-1}R_{e,k-1}^{\lambda}K_{p,k}^{\top} \\
& = F\Diff{P_{k|k-1}}H^{\top}K_{p,k}^{\top}+F\Diff{P_{k|k-1}}H^{\top}K_{p,k-1}^{\top}.
\end{align*}
Finally,
\begin{align}
\Diff{K_{p,k}R_{e,k}^{\lambda}K_{p,k}^{\top}} & = F\Diff{P_{k|k-1}}H^{\top}K_{p,k}^{\top} \nonumber \\
& +F\Diff{P_{k|k-1}}H^{\top}K_{p,k-1}^{\top} \label{proof:eq2} \\
& +K_{p,k-1}R_{e,k-1}^{\lambda}K_{p,k}^{\top}- K_{p,k}R_{e,k}^{\lambda}K_{p,k-1}^{\top}.  \nonumber
\end{align}

Having substituted~\eqref{proof:eq2} into~\eqref{proof:1}, we obtain
\begin{align}
\Diff{P_{k+1|k}} & =  F\Diff{P_{k|k-1}}F^{\top} - \lambda F\Diff{P_{k|k-1}}H^{\top}K_{p,k}^{\top} \nonumber \\
 & - \lambda F\Diff{P_{k|k-1}}H^{\top}K_{p,k-1}^{\top} - \lambda K_{p,k-1}R_{e,k-1}^{\lambda}K_{p,k}^{\top}  \nonumber \\
 & + \lambda K_{p,k}R_{e,k}^{\lambda}K_{p,k-1}^{\top}. \label{lemma:final:P}
\end{align}

The required formulas~\eqref{lemma:1}, \eqref{lemma:2} are both derived from equation~\eqref{lemma:final:P} by substituting the related recursions for $K_{p,k}$ and $R_{e,k}$ and, then, by expanding the terms in the resulted equation. Indeed, we prove~\eqref{lemma:2} by taking into account that $R_{e,k}$ is a symmetric matrix and by substituting~\eqref{proof:new}, \eqref{proof:Kpk}  into~\eqref{lemma:final:P} as follows:
\begin{align}
\Diff{P_{k+1|k}} & =  F\Diff{P_{k|k-1}}F^{\top} \nonumber \\
 & - \lambda F\Diff{P_{k|k-1}}H^{\top}  \left[R_{e,k}^{\lambda} \right]^{-1} R_{e,k-1}^{\lambda} K_{p,k-1}^{\top} \nonumber  \\
 & - \lambda F\Diff{P_{k|k-1}}H^{\top}  \left[R_{e,k}^{\lambda} \right]^{-1} H\Diff{P_{k|k-1}}F^{\top}  \nonumber \\
 & - \lambda F\Diff{P_{k|k-1}}H^{\top}K_{p,k-1}^{\top} - \lambda K_{p,k-1}R_{e,k-1}^{\lambda}K_{p,k}^{\top} \nonumber \\
 & + \lambda \left[K_{p,k-1}R_{e,k-1}^{\lambda} + F\Diff{P_{k|k-1}}H^{\top}\right] K_{p,k-1}^{\top} \nonumber \\
 & =  F\Diff{P_{k|k-1}}F^{\top} \nonumber \\
 & - \lambda F\Diff{P_{k|k-1}}H^{\top}  \left[R_{e,k}^{\lambda} \right]^{-1} R_{e,k-1}^{\lambda} K_{p,k-1}^{\top} \nonumber  \\
 & - \lambda F\Diff{P_{k|k-1}}H^{\top}  \left[R_{e,k}^{\lambda} \right]^{-1} H\Diff{P_{k|k-1}}F^{\top} \nonumber  \\
 & - \lambda K_{p,k-1}R_{e,k-1}^{\lambda}  \left[R_{e,k}^{\lambda} \right]^{-1} R_{e,k-1}^{\lambda}K_{p,k-1}^{\top} \nonumber  \\
 & - \lambda K_{p,k-1}R_{e,k-1}^{\lambda}  \left[R_{e,k}^{\lambda} \right]^{-1}  H\Diff{P_{k|k-1}}F^{\top} \nonumber \\
 & + \lambda K_{p,k-1}R_{e,k-1}^{\lambda} K_{p,k-1}^{\top}. \label{proof:last:1}
\end{align}

Recall, the matrices $R_{e,k}$ and $R_{e,k-1}$ are symmetric as well as the matrix $\Diff{P_{k|k-1}}$. From equation~\eqref{proof:Rekminus1} we have
\begin{equation}
\left[R_{e,k}^{\lambda} \right]^{-1} R_{e,k-1}^{\lambda} = (I - \lambda \left[R_{e,k}^{\lambda} \right]^{-1}  H\Diff{P_{k|k-1}}H^{\top}).
\label{proof:last}
\end{equation}

Having substituted~\eqref{proof:last} into~\eqref{proof:last:1}, we obtain
\begin{align*}
 & \Diff{P_{k+1|k}}  =  F\Diff{P_{k|k-1}}F^{\top} \nonumber \\
 & - \lambda F\Diff{P_{k|k-1}}H^{\top} \left(I - \lambda \left[R_{e,k}^{\lambda} \right]^{-1}  H\Diff{P_{k|k-1}}H^{\top}\right) K_{p,k-1}^{\top}  \\
 & - \lambda F\Diff{P_{k|k-1}}H^{\top}  \left[R_{e,k}^{\lambda} \right]^{-1} H\Diff{P_{k|k-1}}F^{\top}  \\
 & - \lambda K_{p,k-1}R_{e,k-1}^{\lambda}  \left(I - \lambda \left[R_{e,k}^{\lambda} \right]^{-1}  H\Diff{P_{k|k-1}}H^{\top}\right)K_{p,k-1}^{\top} \\
 & - \lambda K_{p,k-1} \left(I - \lambda H\Diff{P_{k|k-1}}H^{\top}\left[R_{e,k}^{\lambda} \right]^{-1} \right) H\Diff{P_{k|k-1}}F^{\top} \\
 & + \lambda K_{p,k-1}R_{e,k-1}^{\lambda} K_{p,k-1}^{\top}.
\end{align*}

Finally, having substituted~\eqref{proof:last} one more time into equation above and, next, having collected the similar terms, we arrive at~\eqref{lemma:2}. The same  approach is used for deriving~\eqref{lemma:1}. The difference with the derivation above is in the replacement of the term $K_{p,k-1}$ by $K_{p,k}$ via the corresponding recursion instead of avoiding the terms $K_{p,k}$ in equation~\eqref{lemma:final:P} and replacing them by $K_{p,k-1}$.
\end{proof}

\begin{remark}
As can be seen, Lemma~1 and, hence, the first Chandrasekhar MCC-KF-type recursions are proved for a constant scalar adjusting parameter $\lambda$. We stress that there exist some adaptive kernel size selection strategies for $\sigma_k$ that yield a constant adjusting weight $\lambda$. A particular example of such adaptive selection rules can be found in~\cite{2016:Izanloo} and the resulted estimator is successfully applied for solving practical application in~\cite{2017:Yang,2018:Yang}. Thus, the new theoretical result in Lemma~1 has a practical interest, although it is restricted to a constant adjusting weight case. The derivation of the Chandrasekhar recursion for a general case of time-varying $\lambda_k$ is complicated for the MCC-KF estimators. This is an open question for a future research.
\end{remark}

Now, we are ready to propose the first Chandrasekhar-based IMCC-KF implementations. Let's consider the first recursion in Lemma~1, i.e. equation~\eqref{lemma:1}. Recall, the goal is to propagate the low rank factors $L_k \in {\mathbb R}^{n\times\alpha}$ and $M_k \in {\mathbb R}^{\alpha \times\alpha}$ of the difference matrix $\Diff{P_{k+1|k}} = P_{k+1|k} - P_{k|k-1}$ where $\Diff{P_{k+1|k}}  = L_kM_kL_k^{\top}$ instead of the full matrix $P_{k+1|k} \in {\mathbb R}^{n\times n}$. Taking into account the required factorization, one may express equation~\eqref{lemma:1} as follows:
\begin{align*}
&\Diff{P_{k+1|k}}  = L_kM_kL_k^{\top} = \Bigl(F-\lambda K_{p,k}H\Bigr)\Bigl(\Diff{P_{k|k-1}}  \\
 &  \phantom{=} +\lambda \Diff{P_{k|k-1}}H^{\top} [R_{e,k-1}^{\lambda}]^{-1} H\Diff{P_{k|k-1}}\Bigr) \\
 & \phantom{=} \times \Bigl(F-\lambda K_{p,k}H\Bigr)^{\top}  = \Bigl(F-\lambda K_{p,k}H\Bigr)\Bigl(L_{k-1}M_{k-1}L_{k-1}^{\top}  \\
 & \phantom{=} +\lambda L_{k-1}M_{k-1}L_{k-1}^{\top}H^{\top} [R_{e,k-1}^{\lambda}]^{-1} HL_{k-1}M_{k-1}L_{k-1}^{\top}\Bigr) \\
 & \phantom{=} \times \Bigl(F-\lambda K_{p,k}H\Bigr)^{\top}  = \Bigl(F-\lambda K_{p,k}H\Bigr)L_{k-1} \\
 & \phantom{=} \times\Bigl(M_{k-1} + \lambda M_{k-1}L_{k-1}^{\top}H^{\top} [R_{e,k-1}^{\lambda}]^{-1} HL_{k-1}M_{k-1}\Bigr) \\
 & \phantom{=} \times L_{k-1}^{\top}\Bigl(F-\lambda K_{p,k}H\Bigr)^{\top}.
\end{align*}

Having compared both sides of the resulted equality, we conclude
\begin{align}
L_k & :=\Bigl(F-\lambda K_{p,k}H\Bigr)L_{k-1}, \label{fast1:Lk} \\
M_k & := M_{k-1} + \lambda M_{k-1}L_{k-1}^{\top}H^{\top} [R_{e,k-1}^{\lambda}]^{-1} HL_{k-1}M_{k-1}. \label{fast1:Mk}
\end{align}

The last step is to express $K_{p,k}$ and $R_{e,k}$ in terms of the factors $L_k$ and $M_k$. From formula~\eqref{proof:Rekminus1}, we have
\begin{align}
R_{e,k}^{\lambda} & =  R_{e,k-1}^{\lambda} + \lambda H\Diff{P_{k|k-1}}H^{\top} \nonumber \\
& =  R_{e,k-1}^{\lambda} + \lambda HL_{k-1}M_{k-1}L_{k-1}^{\top}H^{\top}. \label{fast1:Rek}
\end{align}

At the same way, from equation~\eqref{proof:Kpk} we obtain
\begin{align}
K_{p,k} & = \left[K_{p,k-1}R_{e,k-1}^{\lambda} + F\Diff{P_{k|k-1}}H^{\top}\right] \left[R_{e,k}^{\lambda} \right]^{-1} \nonumber \\
& \left[K_{p,k-1}R_{e,k-1}^{\lambda} + FL_{k-1}M_{k-1}L_{k-1}^{\top}H^{\top}\right] \left[R_{e,k}^{\lambda} \right]^{-1} \label{fast1:Kpk}
\end{align}

Having collected equations~\eqref{fast1:Lk}~-- \eqref{fast1:Kpk} and formula~\eqref{MCCKF:X:Riccati1} used for computing the state estimate, the first Chandrasekhar recursion-based IMCC-KF implementation is designed. For readers' convenience, it is summarized in the form of pseudo-code in Algorithm~1.

\begin{codebox}
\Procname{{\bf Algorithm 1}. $\proc{Chandrasekhar IMCC-KF}$ based on recursion~\eqref{lemma:1}}
\zi \textsc{Initialization:}($k=0$)
\li \>Set $x_{0|-1} = \bar x_0$, $P_{0|-1}=\Pi_0$;
\li \>Compute $R_{e,0}^{\lambda} = R +\lambda H\Pi_0H^{\top}$, $K_{p,0} = F\Pi_0 H^{\top}[R_{e,0}^{\lambda}]^{-1}$;
\li \>Find $\Diff{P_{1|0}} \!= \!F\Pi_0F^{\top}\!+\!GQG^{\top}\! - \lambda K_{p,0}R_{e,0}^{\lambda}K_{p,0}^{\top}\!  - \!\Pi_0$;
\li \>Factorize $\Diff{P_{1|0}} = L_0M_0L_0^{\top}$, $L_0 \in {\mathbb R}^{n\times\alpha}$, $M_0 \in {\mathbb R}^{\alpha\times\alpha}$;
\zi \textsc{Filter Recursion}: ($k=\overline{0,N}$)
\li \>$R_{e,k+1}^{\lambda}  = R_{e,k} +\lambda H L_{k} M_{k}L_{k}^{\top} H^{\top}$;
\li \>$K_{p,k+1}  = \left[K_{p,k}R_{e,k}^{\lambda} + FL_{k} M_{k}L_{k}^{\top}H^{\top}\right] \left[R_{e,k+1}^{\lambda} \right]^{-1}$;
\li \>$L_{k+1}  = (F-\lambda K_{p,k+1}H\Bigr)L_{k}$;
\li \>$M_{k+1}  = M_{k} + \lambda  M_{k}L_{k}^{\top} H^{\top} [R_{e,k}^{\lambda}]^{-1} H L_{k}M_{k}$;
\li \>$\hat x_{k+1|k}  =   F \hat x_{k|k-1} + \lambda K_{p,k}(y_k - H \hat x_{k|k-1})$.
\end{codebox}

It is worth noting here that the error covariance matrix $P_{k+1|k}$ is simply recovered from the propagated factors $L_k \in {\mathbb R}^{n\times\alpha}$ and $M_k \in {\mathbb R}^{\alpha \times\alpha}$ of the matrix $\Diff{P_{k+1|k}} = P_{k+1|k}- P_{k|k-1}$ at any time instance, if necessary:
\[P_{k+1|k} = P_{k|k-1} + L_k M_k L_k^{\top} = \Pi_0 + \sum \limits_{j=0}^{k} L_j M_j L_j^{\top}.\]

Alternatively, the Chandrasekhar recursion in~\eqref{lemma:2} might be utilized for designing the related IMCC-KF implementation as follows:
\begin{align*}
\Diff{P_{k+1|k}}  & = L_kM_kL_k^{\top} =  \Bigl(F-\lambda K_{p,k-1}H\Bigr)L_{k-1}   \\
 & \phantom{=} \times \Bigl(M_{k-1} -\lambda M_{k-1}L_{k-1}^{\top}H^{\top} [R_{e,k}^{\lambda}]^{-1} HL_{k-1}M_{k-1}\Bigr) \\
 & \phantom{=} \times L_{k-1}^{\top}\Bigl(F-\lambda K_{p,k-1}H\Bigr)^{\top},
\end{align*}
i.e. we conclude
\begin{align}
L_k & := \Bigl(F-\lambda K_{p,k-1}H\Bigr)L_{k-1}, \label{fast2:Lk} \\
M_k & := M_{k-1} - \lambda M_{k-1}L_{k-1}^{\top}H^{\top} [R_{e,k}^{\lambda}]^{-1} HL_{k-1}M_{k-1}. \label{fast2:Mk}
\end{align}

Having summarized equations~\eqref{fast1:Rek}~-- \eqref{fast2:Mk} with~\eqref{MCCKF:X:Riccati1}, a new Chandrasekhar-based IMCC-KF implementation is designed.
\begin{codebox}
\Procname{{\bf Algorithm 2}. $\proc{Chandrasekhar IMCC-KF}$ based on recursion~\eqref{lemma:2}}
\zi \textsc{Initialization:}($k=0$)
\li \>Set $x_{0|-1} = \bar x_0$, $P_{0|-1}=\Pi_0$;
\li \>Compute $R_{e,0}^{\lambda} = R +\lambda H\Pi_0H^{\top}$, $K_{p,0} = F\Pi_0 H^{\top}[R_{e,0}^{\lambda}]^{-1}$;
\li \>Find $\Diff{P_{1|0}} \!= \!F\Pi_0F^{\top}\!+\!GQG^{\top}\! - \lambda K_{p,0}R_{e,0}^{\lambda}K_{p,0}^{\top}\!  - \!\Pi_0$;
\li \>Factorize $\Diff{P_{1|0}} = L_0M_0L_0^{\top}$, $L_0 \in {\mathbb R}^{n\times\alpha}$, $M_0 \in {\mathbb R}^{\alpha\times\alpha}$;
\zi \textsc{Filter Recursion}: ($k=\overline{0,N}$)
\li \>$R_{e,k+1}^{\lambda}  = R_{e,k} +\lambda H L_{k} M_{k}L_{k}^{\top} H^{\top}$;
\li \>$K_{p,k+1}  = \left[K_{p,k}R_{e,k}^{\lambda} + FL_{k} M_{k}L_{k}^{\top}H^{\top}\right] \left[R_{e,k+1}^{\lambda} \right]^{-1}$;
\li \>$L_{k+1}  = (F-\lambda K_{p,k}H)L_{k}$;
\li \>$M_{k+1}  = M_{k} - \lambda  M_{k}L_{k}^{\top} H^{\top} [R_{e,k+1}^{\lambda}]^{-1} H L_{k}M_{k}$;
\li \>$\hat x_{k+1|k}  =   F \hat x_{k|k-1} + \lambda K_{p,k}(y_k - H \hat x_{k|k-1})$.
\end{codebox}

Having compared equations~\eqref{fast1:Lk} and~\eqref{fast2:Lk} for computing $L_k$ factor, we observe a difference between Algorithms~1 and~2. In fact,  equation~\eqref{fast1:Lk} of Algorithm~1 implies calculation of $L_k$ through $K_{p,k}$ value obtained at the same filtering step $t_k$. Meanwhile equation~\eqref{fast2:Lk} of Algorithm~2 computes $L_k$ by using $K_{p,k-1}$, i.e. by using the value from the previous filtering step $t_{k-1}$. Besides, a similar difference in calculating $M_k$ factor between Algorithms~1 and~2 is observed from equations~\eqref{fast1:Mk} and~\eqref{fast2:Mk}. Indeed, Algorithm~1 implies calculation of $M_k$ through the previous step value $R_{e,k-1}$, meanwhile equation~\eqref{fast2:Lk} of Algorithm~2 finds $M_k$ by using the current value $R_{e,k}$. Finally, we note that each filtering step of Algorithm~1 require two $m\times m$ matrix inversions, i.e. the terms $\left[R_{e,k}^{\lambda} \right]^{-1}$ and $\left[R_{e,k+1}^{\lambda} \right]^{-1}$. In contrast, Algorithm~2 involves only $\left[R_{e,k+1}^{\lambda} \right]^{-1}$. However, a careful implementation of Algorithm~1 suggests to save the previously obtained $\left[R_{e,k}^{\lambda} \right]^{-1}$ value and, hence, to reduce the number of matrix inversions up to one operation. In this case, the computational complexity of Algorithms~1 and~2 are the same, i.e. the same number of flops are required.
Indeed, Algorithms~1 and~2 differ by the way of calculating $L_k$ and $M_k$ factors, only. Besides, they are computed at a similar ways. Hence, a careful implementation of Algorithm~1 yields the same number of flops.

Next, we consider Algorithm~2 and derive a few more implementations based on recursion~\eqref{lemma:2}. The same can be done with respect to recursion~\eqref{lemma:1}. First, we suggest a variant of Algorithm~2 that requires the inversion of $\alpha$-by-$\alpha$ matrix instead of $R_{e,k+1}$ of size $m$ involved in Algorithm~2. Such implementation might be preferable for practical use, especially when $\alpha << m$, where $m$ is the size of measurement vector $y_k$ in the state-space model. Having applied the matrix inversion lemma (Sherman-Morrison-Woodbury formula) to the recursion in~\eqref{fast1:Rek}, we get
\begin{align*}
[R_{e,k+1}^{\lambda}]^{-1} & =  [R_{e,k}^{\lambda} + \lambda HL_kM_kL_k^{\top}H^{\top}]^{-1} = [R_{e,k}^{\lambda}]^{-1}   \\
& - \lambda [R_{e,k}^{\lambda}]^{-1}HL_k \Bigl(M_k^{-1} + \lambda L_k^{\top}H^{\top}[R_{e,k}^{\lambda}]^{-1}HL_k \Bigr)^{-1} \\
& \times L_k^{\top}H^{\top}[R_{e,k}^{\lambda}]^{-1}.
\end{align*}
Furthermore,  taking into account that
\begin{align*}
& \Bigl(M_k^{-1} + \lambda L_k^{\top}H^{\top}[R_{e,k}^{\lambda}]^{-1}HL_k \Bigr)^{-1} \\
& = M_k - \lambda M_k L_k^{\top}H^{\top} (R_{e,k}^{\lambda} + \lambda HL_k M_k L_k^{\top}H^{\top})^{-1} HL_k M_k \\
& = M_k - \lambda M_k L_k^{\top}H^{\top} (R_{e,k+1}^{\lambda})^{-1} HL_k M_k = M_{k+1},
\end{align*}
we conclude that if $M_k$ is updated through~\eqref{fast2:Mk}, then
\begin{align*}
[R_{e,k+1}^{\lambda}]^{-1} & =  [R_{e,k}^{\lambda}]^{-1}  - \lambda [R_{e,k}^{\lambda}]^{-1}HL_k M_{k+1} L_k^{\top}H^{\top}[R_{e,k}^{\lambda}]^{-1}.
\end{align*}

Having summarized the formulas above, we formulate Algorithm~3, which is mathematically equivalent to Algorithm~2 and through Lemma~1 it is equivalent to Algorithm~1 as well.
\begin{codebox}
\Procname{{\bf Algorithm 3}. $\proc{Chandrasekhar IMCC-KF}$ based on recursion~\eqref{lemma:2}}
\zi \textsc{Initialization:}($k=0$)
\li \>Set $x_{0|-1} = \bar x_0$, $P_{0|-1}=\Pi_0$;
\li \>Compute $R_{e,0}^{\lambda} = R +\lambda H\Pi_0H^{\top}$, $K_{0} = F\Pi_0 H^{\top}$;
\li \>Find $\Diff{P_{1|0}} \!= \!F\Pi_0F^{\top}\!+\!GQG^{\top}\! -\lambda K_{0}[R_{e,0}^{\lambda}]^{-1}K_{0}^{\top}\!  - \!\Pi_0$;
\li \>Factorize $\Diff{P_{1|0}} = L_0M_0L_0^{\top}$, $L_0 \in {\mathbb R}^{n\times\alpha}$, $M_0 \in {\mathbb R}^{\alpha\times\alpha}$;
\zi \textsc{Filter Recursion}: ($k=\overline{0,N}$)
\li \>$L_{k+1}  = (F-\lambda K_{k}[R_{e,k}^{\lambda}]^{-1}H\Bigr)L_{k}$;
\li \>$M_{k+1}^{-1}  = M_k^{-1} + \lambda L_k^{\top}H^{\top}[R_{e,k}^{\lambda}]^{-1}HL_k$;
\li \>$[R_{e,k+1}^{\lambda}]^{-1}  =  [R_{e,k}^{\lambda}]^{-1}$ \label{line:3:M}
\zi \>$\phantom{[R_{e,k+1}^{\lambda}]^{-1}  = }- \lambda [R_{e,k}^{\lambda}]^{-1}HL_k M_{k+1} L_k^{\top}H^{\top}[R_{e,k}^{\lambda}]^{-1}$;
\li \>$K_{k+1}  = K_{k} + FL_{k} M_{k}L_{k}^{\top}H^{\top}$;
\li \>$\hat x_{k+1|k}  =   F \hat x_{k|k-1} + \lambda K_{k}[R_{e,k}^{\lambda}]^{-1}(y_k - H \hat x_{k|k-1})$.
\end{codebox}

Having analyzed the implementation above, we conclude that Algorithm~3 propagates the inverse matrices $M_{k}^{-1}$ and $[R_{e,k}^{\lambda}]^{-1}$. However, the value $M_k$ is required in line~\ref{line:3:M} of Algorithm~3, i.e. it still demands the matrix inversion operation, but the matrix to be inverted is of size $\alpha$. In other words, if $\alpha << m$, then Algorithm~3 is preferable for practical use from the computational complexity point of view. Meanwhile, the numerical stability issues of Algorithms~1~-- 3 should be also taken into account  when one decides between Algorithms~1, 2 and~3. The numerical behaviour depends on properties of the matrices to be inverted, i.e. on the condition numbers of $R_{e,k}$ in Algorithms~1, 2 and $M_k$ in Algorithm~3.

Finally, we may suggest a symmetric implementation based on Chandrasekhar recursion~\eqref{lemma:2} by re-formulating Algorithm~3.

\begin{codebox}
\Procname{{\bf Algorithm 4}. $\proc{Chandrasekhar IMCC-KF}$ based on recursion~\eqref{lemma:2}}
\zi \textsc{Initialization:}($k=0$)
\li \>Set $x_{0|-1} = \bar x_0$, $P_{0|-1}=\Pi_0$;
\li \>Compute $R_{e,0}^{\lambda} = R +\lambda H\Pi_0H^{\top}$, $K_{p,0} = F\Pi_0 H^{\top}[R_{e,0}^{\lambda}]^{-1}$;
\li \>Find $\Diff{P_{1|0}} \!= \!F\Pi_0F^{\top}\!+\!GQG^{\top}\! - \lambda K_{p,0}R_{e,0}^{\lambda}K_{p,0}^{\top}\!  - \!\Pi_0$;
\li \>Factorize $\Diff{P_{1|0}} = L_0M_0L_0^{\top}$, $L_0 \in {\mathbb R}^{n\times\alpha}$, $M_0 \in {\mathbb R}^{\alpha\times\alpha}$;
\zi \textsc{Filter Recursion}: ($k=\overline{0,N}$)
\li \>$R_{e,k+1}^{\lambda} =  R_{e,k} +\lambda H L_{k} M_{k}L_{k}^{\top} H^{\top}$; \label{alg4:Rek}
\li \>$M_{k+1}^{-1}  = M_k^{-1} + \lambda L_k^{\top}H^{\top}[R_{e,k}^{\lambda}]^{-1}HL_k$;  \label{alg4:Mk}
\li \>$K_{p,k+1}  =  \left[K_{p,k}R_{e,k}^{\lambda} + FL_{k} M_kL_{k}^{\top}H^{\top}\right] \left[R_{e,k+1}^{\lambda} \right]^{-1}$;
\li \>$L_{k+1}  = (F-\lambda K_{p,k}H\Bigr)L_{k}$;
\li \>$\hat x_{k+1|k}  =   F \hat x_{k|k-1} + \lambda K_{p,k}(y_k - H \hat x_{k|k-1})$.
\end{codebox}

From the numerical stability and computational complexity reasons, Algorithm~4 seems to be the worst implementation, although the formulas in lines~\ref{alg4:Rek}, \ref{alg4:Mk} have a symmetric form. Indeed, Algorithm~4 involves both the $m \times m$ and $\alpha \times \alpha$ matrix inversions in each iterate.

Finally, the required low-rank approximation for $\Diff{P_{1|0}}$ and the displacement rank $\alpha$ are discussed.
In contrast to a symmetric positive (semi-) definite error covariance matrix $P_{k|k-1}$ involved in the Riccati recursion, the matrix $\Diff{P_{k|k-1}}$ is a symmetric indefinite matrix. Following~\cite[Chapter~13]{2000:Kailath:book}, the required factorization $\Diff{P_{1|0}} = L_0M_0L_0^{\top}$, $L_0 \in {\mathbb R}^{n \times \alpha}$, $M_0 \in {\mathbb R}^{\alpha \times \alpha}$ can be performed in various ways; e.g. by using Bunch-Kaufman algorithm~\cite{1977:Bunch,1998:Ashcraft}. More precisely, the accurate Bunch-Kaufman method from~\cite{2002:Higham:book} with corresponding MATLAB routine \verb"ldl" is utilized for implementing Algorithms~1--4. Given a symmetric indefinite matrix $A \in {\mathbb R}^{n \times n}$, it performs factorization $PAP^{\top} = LDL^{\top}$ with  a permutation
$P \in {\mathbb R}^{n \times n}$, a unit lower triangular $L \in {\mathbb R}^{n \times n}$, and a real block diagonal $D \in {\mathbb R}^{n \times n}$. Thus, the displacement rank $\alpha$ equals the rank of the resulted matrix $D$. We stress that it is defined automatically and it heavily depends on initial value $P_0$ and the problem statement. For instance, if $\Pi_0 = 0$, then $K_{p,0} = 0$, $R_{e,0}^{\lambda} = R$ and $\Diff{P_{1|0}} = GQG^{\top}$, i.e. one may set $L_0:= G \in {\mathbb R}^{n \times q}$ and $M_0:= Q \in {\mathbb R}^{q \times q}$, i.e. $\alpha:=q$. In practice, the process covariance matrix is often a diagonal matrix with a few non-zero diagonal elements (that is $q$) and, hence, in this case $\alpha << n$. In general, the displacement rank satisfies $\alpha \le n$. For any initial value $\Pi_0$, the value of $\alpha$ is defined and, next, the resulted block diagonal matrix $D$ of size $n$ is approximated by its part of size $\alpha$ with corresponding non-zero (block) diagonal elements, i.e. we set $M_0:= [D]_{\alpha}  \in {\mathbb R}^{\alpha\times \alpha}$. Finally, the product $P^{\top}L$ is computed  and, next, the columns that correspond to non-zero elements in $D$ are pulled out to get an approximation $L_0:=[P^{\top}L] \in {\mathbb R}^{n\times \alpha}$. It should be stressed that the required factorization is performed only once, i.e. at the initial filtering step and, next, all fast Morf-Sidhu-Kailath-Sayed implementations propagate the resulted factors $L_k \in {\mathbb R}^{n\times\alpha} $ and $M_k \in {\mathbb R}^{\alpha \times\alpha} $ according to the underlying Chandrasekhar recursions.

\begin{table*}
\caption{The RMSE, the average CPU time (s) and runtime benefit of the Chandrasekhar methods over the Riccati filtering in Example~\ref{ex:1}.} \label{tab:numeric}
\centering
{\small
\begin{tabular}{ll||cccc|c||cccc|c|cc}
\hline
 & & \multicolumn{5}{c||}{$n=4$, $\Pi_0 = diag\{[1, 1, 1, 10^{-2}]\}$, $\alpha =4$} &
 \multicolumn{7}{c}{$n=4$, $\Pi_0 = 0$, $\alpha =1$} \\
\cline{3-14}
Ratio & KF and IMCC-KF & \multicolumn{5}{c||}{Root mean square errors $\mbox{RMSE}_{x_i}$} &
\multicolumn{5}{c|}{Root mean square errors $\mbox{RMSE}_{x_i}$} & \multicolumn{2}{c}{CPU time}  \\
\cline{3-14}
\cline{3-14}
$q_4/R$ & implementations  & $x_1$ & $x_2$ & $x_3$ & $x_4$ & $\|\cdot\|_2$
                & $x_1$ & $x_2$ & $x_3$ & $x_4$ & $\|\cdot\|_2$ &  (s.) & (\%)  \\
\hline
$0.63 \cdot 10^{-2}$ & Classical Riccati KF
            & 80.95 & 1.77   &  0.42 &	0.92  &	80.97
            & 78.37 & 2.07   &  0.00 &  0.91  & 78.40 &  0.0159 & NA   \\
\cline{2-14}
& IMCC-KF Riccati~\eqref{MCCKF:X:Riccati1}-\eqref{MCCKF:P:Riccati1}
            & 80.90 &  1.95  &  0.42 &	  0.93 &    80.93
            & 77.69 &  2.34  &  0.00 &	  0.91 &    77.73 &  0.0212 & -   \\
& IMCC-KF Algorithm~1
            & 80.90 &  1.95  &  0.42 &	  0.93 &    80.93
            & 77.69 &  2.34  &  0.00  &	  0.91 &    77.73 &  0.0203 & 4.4   \\
& IMCC-KF Algorithm~2
            & 80.90 &  1.95  &  0.42 &	  0.93 &    80.93
            & 77.69 &  2.34  &  0.00  &	  0.91 &    77.73 &  0.0202 & 5.2   \\
& IMCC-KF Algorithm~3
            & 80.90 &  1.95  &  0.42 &	  0.93 &    80.93
            & 77.69 &  2.34  &  0.00  &	  0.91 &    77.73 &  0.0205 & 3.7   \\
& IMCC-KF Algorithm~4
            & 80.90 &  1.95  &  0.42 &	  0.93 &    80.93
            & 77.69 &  2.34  &  0.00  &	  0.91 &    77.73 &  0.0206 & 3.2   \\
\hline
\hline
$0.63 \cdot 10^{-4}$ & Classical Riccati KF
            & 81.77 & 4.26   &  0.44 &  0.93  & 81.89
            & 60.75 & 5.93   &  0.00 &  0.92  & 61.05 & 0.0167 & NA   \\
\cline{2-14}
& IMCC-KF Riccati~\eqref{MCCKF:X:Riccati1}-\eqref{MCCKF:P:Riccati1}
            & 81.32 &  3.95  &  0.44  &	  0.93 &	81.42
            & 56.54 &  6.61  &  0.00 &	  0.92 &	56.93 & 0.0218 & -   \\
& IMCC-KF Algorithm~1
            & 81.32 &  3.95  &  0.44  &	  0.93 &	81.42
            & 56.54 &  6.61  &  0.00 &	  0.92 &	56.93 & 0.0210 & 3.6   \\
& IMCC-KF Algorithm~2
            & 81.32 &  3.95  &  0.44  &	  0.93 &	81.42
            & 56.54 &  6.61  &  0.00 &	  0.92 &	56.93 & 0.0206 & 6.0   \\
& IMCC-KF Algorithm~3
            & 81.32 &  3.95  &  0.44  &	  0.93 &	81.42
            & 56.54 &  6.61  &  0.00 &	  0.92 &	56.93  & 0.0211 & 3.5   \\
& IMCC-KF Algorithm~4
            & 81.32 &  3.95  &  0.44  &	  0.93 &	81.42
            & 56.54 &  6.61  &  0.00 &	  0.92 &	56.93  & 0.0214 & 1.9   \\
\hline
\end{tabular}
}
\end{table*}

\section{Numerical Experiments}

The goal is to justify the theoretical derivation of the newly suggested Chandrasekhar-based MCC-KF estimators in Section~\ref{sec:derivation}.

\begin{exmp} \label{ex:1}
 The dynamic  of the in-track motion  of a satellite traveling in a circular orbit is given as follows~\cite[p.~1448]{1965:Rauch}:
\begin{align*}
x_{k+1} & =
\begin{bmatrix}
1 & 1 & 0.5 &  0.5 \\
0 & 1 & 1 & 1 \\
0 & 0 & 1 & 0 \\
0 & 0 & 0 & 0.606
\end{bmatrix}
x_{k}\! +\! w_{k},
Q
=
\begin{bmatrix}
0 & 0 & 0 & 0 \\
0 & 0 & 0 & 0 \\
0 & 0 & 0 & 0 \\
0 & 0 & 0 & q_{4}
\end{bmatrix},
 \\
y_k & =
\begin{bmatrix}
1 & 0 & 0 & 0
\end{bmatrix}
x_k + v_k, \quad R = 1,
\end{align*}
with zero-mean initial state and $\Pi_0 = diag\{[1, 1, 1, 10^{-2}]\}$. In~\cite{1965:Rauch}, two cases are examined: (i) $q_4=0.63 \cdot 10^{-2}$, and (ii)  $q_4=0.63 \cdot 10^{-4}$.
\end{exmp}

The original MCC-KF estimator has been derived for dealing with the impulsive noise case in~\cite{2016:Izanloo}. Here, we follow the same experimental conditions as in the cited paper, i.e.
\begin{align*}
w_k  & \sim {\cal N}(0, Q)+\mbox{\tt Shot noise}, \\
v_k  & \sim {\cal N}(0, R)+\mbox{\tt Shot noise}
\end{align*}
where the short noise is generated as follows: (i) only 10\% of samples are corrupted by the outliers; (ii) the discrete time instants $t_k$ corrupted are selected randomly from the uniform discrete distribution in $[21,N-1]$ where the system is simulated for $N=300$ discrete time points; (iii) the magnitude of each impulse is chosen randomly from the uniform discrete distribution in the interval $[0,3]$.

When the stochastic model is simulated for $N=300$, the inverse problem (i.e. the state estimation from the observed signal) is solved by various filtering methods under examination. The root mean square error (RMSE) is calculated over $M=500$ Monte Carlo runs to justify the estimation accuracy. Additionally, the average CPU time (s) is collected for each implementation in Table~\ref{tab:numeric}. The runtime benefit is also computed for Algorithms~1--4 compared to the Riccati-based counterpart as follows: $(CPU_{Chandrasekhar}/CPU_{Riccati}-1)$ and it is expressed in percent. It is worth noting here that the CPU time benefit is remarkable when $\alpha << n$. It is not difficult to see that if $\Pi_0=0$ in Example~1, then $\Diff{P_{1|0}} = GQG^{\top}$ where $G = I_4$ and, hence, $\alpha = 1$ while $n = 4$. For this case we present the CPU time and the computed computational benefit in Table~\ref{tab:numeric}. In general case of $\Pi_0$ utilized in Example~\ref{ex:1}, we do not know in advance the exact value of $\alpha$, because it is defined automatically from the related factorization. However, for readers' convenience, we force Algorithms~1--4 to return this value and summarize the results for various $\Pi_0$ together with the outcomes of the numerical experiments in Table~\ref{tab:numeric}.

Having analyzed the results summarized in Table~\ref{tab:numeric}, we make a few conclusions. First, we observe that the MCC-KF implementations outperform the classical KF for estimation quality in case of impulsive noise examined in this paper. This result was anticipated, because the MCC KF-like estimators are shown to be more robust with respect to outliers than the classical KF methodology in many recent papers. In our experiments, the difference in estimation accuracies between the classical KF and the MCC KF-like methods is modest, because we use the adaptive kernel size selection approach suggested in~\cite{2016:Izanloo} that yields the constant adjusting parameter $\lambda$.This allows for testing the suggested Chandrasekhar implementations. For other adaptive kernel size selection strategies the difference in estimation quality between the classical KF and the MCC KF-like estimators might be more impressive. However, this paper is rather focused on the existence of the Chandrasekhar-type recursion under the MCC methodology.

Having compared the Riccati- and Chandrasekhar-based MCC KF-like implementations, we conclude that they produce the same estimates of the state vector, i.e. their resulted accuracies are the same for any initial value $\Pi_0$ and various signal-to-noise-ratio values under examination. This substantiates the correctness of the theoretical derivation presented in Section~\ref{sec:derivation}. In other words, the mathematical equivalence between the classical Riccati- and the newly suggested Chandrasekhar-based filtering under the MCC-KF methodology is validated in practice.

Finally, having compared the CPU time averaged over Monte Carlo runs, we conclude that the Chandrasekhar-based implementations are faster than their algebraically equivalent Riccati-based counterparts when $\alpha < n$. This is in line with the previous research focused on the Chandrasekhar-based KF implementations in~\cite{1974:Morf,1994:Sayed}. More precisely, from the last column in Table~\ref{tab:numeric} we observe that the Chandrasekhar-based MCC-KF implementations work approximately on $5\%$ faster than the Riccati-based MCC-KF algorithm in case of low-dimensional problem in Example~1, i.e. when $\alpha=1$ and $n=4$.
Thus, we conclude that the newly-developed Chandrasekhar-based MCC-KF implementations provide the same estimation quality to that of the original Riccati-based IMCC-KF algorithm, but at a reduced CPU time. This runtime difference is expected to be significant in case of estimating the large-scale dynamical systems.

\section{Concluding remarks}

In this technical note, the first Chandrasekhar recursion is derived for the maximum correntropy Kalman filtering. The theory is inferred for a case of the adaptive kernel size selection mechanism that yields a constant adjusting parameter. Several Chandrasekhar-based MCC KF-like implementations have been proposed, and the numerical experiments substantiate their practical feasibility and efficiency. The derivation of the Chandrasekhar-type recursion for a general case of a time-variant adjusting weight in the MCC-KF estimators is rather complicated. This is an area for a future research. Another open question is a design of the Kalman-like nonlinear filtering methods under the MCC methodology; e.g. the MCC-EKF is planned for a future work based on the results in~\cite{2016:Kulikov:SISCI,2017:Kulikov:ANM,2017:Kulikov:SP}.


\end{document}